\documentclass{article}

\usepackage[french,british]{babel}

\usepackage[T1]{fontenc}

\usepackage{amsmath}

\usepackage{amsfonts}
\usepackage{amsbsy}
\usepackage{amssymb}

\usepackage{amsthm}

\theoremstyle{plain}

\usepackage{latexsym}
\usepackage{amsmath}
\usepackage{amsfonts}
\usepackage{amssymb}
\usepackage{psfrag}
\usepackage{graphicx}

\usepackage{amsmath,amssymb}
\usepackage{graphicx}

\newtheorem{nas}{Corrolary}[section]

\theoremstyle{Remark}

\theoremstyle{Proposition}

\newtheorem{zau}{Remark}[section]
\theoremstyle{definition}

\newtheorem{proposition}{Proposition}[section]

\newcommand{\lp}[1]{\left( \begin{array}{#1} }
\newcommand{\rp}{\end{array} \right)}

\newcommand{\be}{\begin{equation}}
\newcommand{\ee}{\end{equation}}

\begin{document}

\title{On Distributions of One Class of Random Sums and their Applications}

\author
{Ivan Matsak\thanks
{ Department of Operation Research,
Taras Shevchenko National University of Kyiv, Kyiv 01601, Ukraine
idubovetska@gmail.com },
Mikhail Moklyachuk\thanks
{Department of Probability Theory, Statistics and Actuarial
Mathematics, Taras Shevchenko National University of Kyiv, Kyiv 01601, Ukraine, Moklyachuk@gmail.com},
}

\date{\today}

\maketitle

\renewcommand{\abstractname}{Abstract}
 \begin{abstract}
We propose results of the investigation of properties of the random sums of random variables.
We consider the case, where the number of summands is the first moment of an event occurrence.
An integral equation is presented that determines  distributions of random sums.
With the help of the obtained results we analyse the distribution function of the time
during which the Geiger-Muller counter will not lose any particles, the distribution function of the busy period of a redundant system with renewal,
and the distribution function of the sojourn times of a single-server queueing system.
\end{abstract}

\vspace{2ex}

\maketitle

\vspace{2ex}

\section{Introduction}
 Let  $ \zeta_i,  i \geq 1, $ be independent identically distributed (i.i.d.) random variables,
Let $ \nu $ be a discrete random variable that takes integer positive values. Consider the random sum
\begin{equation} \label{f1}
S _ {\nu} = \sum_{i = 1}^{\nu} \zeta_i.
\end{equation}
In many applied problems of reliability theory, queueing theory, some statistical problems of physics and biology,
there is a need to find the distribution of the random variable $ S_{\nu} $ or at least its main characteristics.
If the random variable $ \nu $ is a Markov moment, then the most famous result here is Wald's identity
\begin{equation}\label{f2}
\mathbf {E} S _ {\nu} = \mathbf{E} \nu \,\mathbf{E} \zeta_1,
  \end{equation}
(see \cite{rch}, \cite{ks}, \cite{fe1}, \cite{kol}, \cite{val}, \cite{gut}, \cite{kalash},  \cite{mik}, \cite{res}
where you can also find a number of related
results and applications).

Similar problems that arise in the mathematical theory of reliability,
were discussed on lectures delivered by B.~V.~Gnedenko at Taras Shevchenko National University of Kyiv, back in
 80-th years of the last century, as well as in a number of his works (see, for example, his preface to the book \cite{kk}).

The case where the random variable $ \nu $ and the sequence $ (\zeta_i) $ are independent, is
studied in details (see \cite{fe1}, \cite{kk}). Unfortunately
in practice quite often the random variable $ \nu $ and the sequence $ (\zeta_i) $ are dependent, which significantly
complicates the problems.

In this paper we consider one important case of such dependence,
where $ \nu $ is the first moment of an event occurrence.
In fact ours analysis is a definite generalization of the methods developed by B.~V.~Gnedenko in his
works \cite{gn1}, \cite{gn2} when studying the problem of reliability of redundant systems with renewal.

\section{Distribution of random sums. Main proposition}
Let $\zeta$ and $\epsilon $ be random variables such that
 \[
\mathbf{P}(\epsilon =1)  = q , \quad \mathbf{P}(\epsilon =0)  = 1- q
, \quad 0 < q < 1,
\]
 \[
\mathbf{P}(\zeta \geq 0) =1, \quad \mathbf{P}(\zeta = 0) < 1.
\]
In general case the random variables $\zeta$ and $\epsilon $ depend on each other.

\noindent Consider the sequence $(\zeta_n, \epsilon_n )$  of independent copies of
$(\zeta, \epsilon )$. Define a random variable $\nu$ in the following way
\[
\nu = \min(n\geq 1 :  \quad \epsilon_{n}= 1).
\]
It is well known, that the random variable $\nu $ has the geometrical distribution
\begin{equation} \label{f3}
\mathbf{P} (\nu = n)= q (1- q )^{n-1} ,  \,  n\geq 1 ,
\end{equation}
and
 \[
 \mathbf{E}\nu =
\frac{1}{q} , \quad  \mathbf{D}\nu = \frac{1-q}{q^2}.
\]

\noindent Consider the random variable $ S_{\nu} $ determined by equality (\ref{f1}).

\noindent We introduce the following notations
\begin{eqnarray}\label{f4}
F_{\zeta, 0} (t) & = & \mathbf{P} (\zeta < t, \, \epsilon =0)
 ,\quad F_{\zeta, 1} (t)  =  \mathbf{P} (\zeta < t, \,  \epsilon =1), \nonumber \\
F_{\zeta} (t) & = & F_{\zeta, 0} (t) + F_{\zeta, 1} (t) =
\mathbf{P}(\zeta < t ) ,\nonumber \\
F_{S} (t) & = & \mathbf{P}( S_{\nu} < t ), \quad P_{S} (t) = 1 -
F_{S} (t) .
\end{eqnarray}
Let
\[
 \psi(z) = \int_0^{\infty} \exp(-zt)dF_{\zeta} (t), \quad \psi_0 (z) = \int_0^{\infty} \exp(-zt)dF_{\zeta, 0} (t),
\]
\[
 \varphi(z) = \int_0^{\infty} \exp(-zt)dF_{S} (t).
\]
be the Laplace transforms of the corresponding distribution functions.

\noindent Introduce the notations
\[
 a =  \mathbf{E}\zeta , \quad  \sigma^2 =  \mathbf{D}\zeta , \quad
a_0 =  \mathbf{E}\zeta I(\epsilon =0) =  \int_0^{\infty} tdF_{\zeta,
0} (t) ,
\]
where   $I(A)$ is the indicator of the event $A$.

\begin{proposition}\label{l1}
(i).  The function  $P_{S} (t) $ satisfies the integral equation
\begin{equation} \label{f5}
P_{S} (t)  = 1-F_{\zeta} (t) +   \int_0^{t} P_{S} (t - x) dF_{\zeta,
0} (x).
\end{equation}

(ii). The Laplace transform  $\varphi(z)$  satisfies the equality
\begin{equation} \label{f6}
\varphi(z) = \frac{ \psi(z) - \psi_0(z)}{ 1 - \psi_0(z)} ;
\end{equation}

(iii).  Under the condition $ \sigma^2 < \infty $ we have the following relations
\begin{equation} \label{f7}
\mathbf{E}S_{\nu}  = \frac{a}{q} ,
\end{equation}
\begin{equation} \label{f8}
 \mathbf{D}S_{\nu}  = \frac{\sigma^2 }{q} + \frac{a^2(q-1)+2a a_0 }{q^2 }.
\end{equation}
\end{proposition}

\begin{proof} (i).  Let
\[
A = \{ S_{\nu} \geq t \}, \quad A_1 = \{ \zeta_{1} \geq t \} ,
\]
\[
 \quad A_2 = \{ \zeta_{1} < t , \quad \epsilon_1 =0 , \quad \sum_{i=2}^{\nu}
 \zeta_i \geq t -  \zeta_{1}  \} .
\]
Then
\[
A = A_1 \bigcup A_2  , \quad A_1 \bigcap A_2 = \emptyset,
\]
and
\begin{equation} \label{f9}
\mathbf{P}(A) = \mathbf{P}(A_1 ) + \mathbf{P}(A_2) .
\end{equation}
The first term in equality (\ref{f9}) is quite simple
\begin{equation} \label{f10}
 \mathbf{P}(A_1 ) = 1 - F_{\zeta} (t) .
\end{equation}
Let $\nu^{'}=  \min(n\geq 2 :  \, \epsilon_{n}= 1) $
and we know that event  $\{ \epsilon_1 =0 \}$ occurred. Then
$\nu = \nu^{'}$.
The random variable $ \nu^{'}$ does not depend on the random variables $ \epsilon_1 $ and $ \zeta_1 $.
Therefore, the relation
\[
\mathbf{P}(\sum_{i=2}^{\nu}
 \zeta_i \geq y / \epsilon_1 =0 , \, \zeta_{1} =x ) = \mathbf{P}(\sum_{i=2}^{\nu^{'}}
 \zeta_i \geq y / \epsilon_1 =0 , \, \zeta_{1} =x )
 =  \mathbf{P}(\sum_{i=2}^{\nu^{'}} \zeta_i \geq y )
\]
\[
 =  \mathbf{P}(\sum_{i=1}^{\nu}
 \zeta_i \geq y  ) = P_{S} (y).
\]
holds true and we have
\begin{eqnarray}\label{f11}
 \mathbf{P}(A_2 ) & = & \int_0^t \mathbf{P}(\sum_{i=2}^{\nu}
 \zeta_i \geq t-x / \epsilon_1 =0 , \, \zeta_{1} =x ) dF_{\zeta,0}(x)
 \nonumber \\
 & = & \int_0^t P_{S}(t-x) dF_{\zeta,0}(x)
 .
\end{eqnarray}
Equalities (\ref{f9}) -- (\ref{f11}) together give equality (\ref{f5}).

(ii).  Equality (\ref{f6})  follows from (\ref{f5}), if we use properties of the Laplace transform.
Really, in terms of the Laplace transform equation
 (\ref{f5}) may be written in the form
\[
 \frac{1}{z} - \frac{\varphi(z)}{z} = \frac{1}{z} -  \frac{\psi(z)}{z}+  \left(\frac{\psi_0(z)}{z} - \frac{\varphi(z)\psi_0(z)}{z} \right),
\]
which gives (\ref{f6}) (similar reasoning can be found in \cite{gbc}, p.331 ).

(iii). Equality (\ref{f7}) is a particular case of Wald's identity (see also relation (\ref{f3})).

Now prove (\ref{f8}). Differentiating equation (\ref{f6}) twice we get
\begin{eqnarray*}
\varphi^{''}(z) & = &  \frac{(\psi^{''}(z)- \psi_0^{''}(z)
)(1-\psi_0(z)) +  \psi_0^{''}(z)(\psi(z)- \psi_0(z)) } {(1-
\psi_0(z))^2} \nonumber \\
 & + & \frac{2(1- \psi_0(z))\psi_0^{'}(z)
\left((\psi^{'}(z)- \psi_0^{'}(z))(1-\psi_0(z)) +
\psi_0^{'}(z)(\psi(z)- \psi_0(z)) \right)} {(1- \psi_0(z))^4} .
  \nonumber \\
  \end{eqnarray*}
Take here $z=0$ and use the know equalities
\[
 \psi_0(0)=  \mathbf{P} (\epsilon = 0)= 1- q ,   \quad  \psi^{'}(0) = -
 \mathbf{E}\zeta  = - a , \quad  \psi_0^{'}(0) = - a_0 ,
 \]
we will have
\[
 \varphi^{''}(0)=  \frac{\psi^{''}(0)}{q} +  \frac{2\psi^{'}(0) \psi_0^{'}(0)  }{q^2}
 = \frac{\sigma^2 +a^2 }{q} +  \frac{2a a_0  }{q^2} .
 \]
If we additionally use the relations
\[
\mathbf{E}S_{\nu}^2  =  \varphi^{''}(0) , \quad \mathbf{D}S_{\nu} =
\mathbf{E}S_{\nu}^2  -
 \frac{a^2 }{q^2} ,
 \]
then we  get (\ref{f8}).
\end{proof}

The following Corollary is very useful in applications (see  \cite{gbc}, p.333)

\begin{nas}\label{n1}
Let under the conditions of Proposition \ref{l1} the quantity  $ a =  \mathbf{E}\zeta < \infty$ is fixed and let
\[
q =  \mathbf{P}(\epsilon =1) \rightarrow 0 .
  \]
Then
\begin{equation} \label{f12}
\lim_{q\rightarrow 0} \mathbf{P}(q S_{\nu} < t )= 1-
\exp\left(-\frac{t}{a} \right) .
\end{equation}
\end{nas}

\begin{proof}
If
\[
\psi_{1}(z)=  \psi(z) - \psi_{0}(z) =  \int_0^{\infty}
\exp(-zt)dF_{\zeta, 1} (t) ,
  \]
then
\[
\psi_{1}(0)=  \mathbf{P}(\epsilon =1) =q .
  \]
  From the estimate
\[
 1- \exp(-x) \leq x, \quad x >0
  \]
we get
\[
\psi_{1}(0) - \psi_{1}(qz) =  \int_0^{\infty} (1- \exp(-qzt))\,
dF_{\zeta, 1} (t) \leq qz \int_0^{\infty} t dF_{\zeta, 1} (t)
  \]
\[
\leq qz \left[K   \int_0^{K}  dF_{\zeta, 1} (t) + \int_K^{\infty} t
dF_{\zeta, 1} (t)  \right] \leq qz  \left[K q + \mathbf{E}\zeta
I(\zeta>K) I(\epsilon = 1) \right].
  \]
Inserting   $K= \frac{1}{\sqrt{q}}$, we get
\[
0\leq \psi_{1}(0) - \psi_{1}(qz) \leq   qz (\sqrt{q} +o(1)) .
\]
That is why
\begin{equation} \label{f13}
\lim_{q\rightarrow 0} \frac{\psi_{1}(qz)}{q}= 1
\end{equation}
uniformly on $z \in (0, C), \, \forall C >0$.

Making use of the relations
\[
\mathbf{P}(q S_{\nu} < t ) = F_S \left(\frac{t}{q}\right),  \quad
\int_0^{\infty} \exp(-zt)\, dF_S \left(\frac{t}{q}\right) =\varphi(qz),
  \]
and the convergence of the Laplace transform
(\cite{fe2}, p.431), we come to conclusion that in order to prove equality (\ref{f12}) it is sufficient to  show that
\begin{equation} \label{f14}
\lim_{q\rightarrow 0} \varphi(qz) = \frac{1}{1+az} .
\end{equation}
We have
\begin{equation} \label{f15}
 \varphi(qz) = \frac{\psi(qz)-\psi_0 (qz)}{1-\psi_0 (qz)} =
 \frac{\psi_1 (qz)/q  }{(1-\psi (qz))/q + \psi_1 (qz)/q } .
\end{equation}
From relation
\[
\lim_{q\rightarrow 0} \frac{(1-\psi (qz))}{qz} = -\psi^{'}(0) = a
\]
and relations  (\ref{f13}) and (\ref{f15}), relation (\ref{f14}) follows.
\end{proof}

\begin{zau}\label{z3}
In the recent paper \cite{ok_ik} the limit theorems for
some regenerative processes were proved based on asymptotic relations of the \eqref{f12} type.
The asymptotic relations \eqref{f12} can be applied when investigating
the length of the queue in stochastic networks (see, for example, \cite{kle}, \cite{ged}, \cite{ged2}).
\end{zau}

\begin{nas}\label{n2}
Let $ \tau  $  and $ \eta  $  be independent random variables, let
\[
F(t)= \mathbf{P}(\tau < t) = 1- \exp(-\lambda t), t \geq 0 , \quad
G(t)= \mathbf{P}(\eta < t),
  \]
\[
\epsilon= I(\tau <  \eta) ,  \quad  \mathbf{P}(\epsilon =1)  = q ,
\quad   0 < q < 1.
  \]
  Denote by $ (\tau_i , \eta_i , \epsilon_i )  $ a sequence of independent copies of  $ (\tau,\eta,\epsilon)$,
\[
\nu = \min(n\geq 1 :  \quad \epsilon_{n}= 1),
\]
\[
 S_{\nu}  = \sum_{i=1}^{\nu} \min(\tau_i , \eta_i )
.  \]
Then
\begin{equation} \label{f16}
 P_S (t) = \mathbf{P}( S_{\nu} \geq t )=
\exp(-\lambda {t} ) .
\end{equation}
\end{nas}

\begin{proof}
If we take  $\zeta =\min(\tau ,\eta ),$ then in notations of Proposition \ref{l1} we have
\[
1- F_{\zeta}(t) = \mathbf{P}(\min(\tau , \eta  ) \geq t) =
\exp(-\lambda t) (1- G(t)),
\]
\[
 F_{\zeta , 0}(t) = \mathbf{P}(\min(\tau , \eta  ) < t , \tau \geq
\eta ) = \int_0^t \exp(-\lambda x) dG(x).
\]
Substituting these functions into equation
(\ref{f5}) we get
\begin{equation} \label{f17}
P_{S} (t)  = \exp(-\lambda t)(1-G(t)) +   \int_0^{t} P_{S} (t - x)
\exp(-\lambda x) dG(x).
\end{equation}
By Proposition \ref{l1} the Laplace transform of the function $F_S (t)= 1-  P_{S}(t) $
is represented by formula (\ref{f6}).
We have to find the functions  $\psi(z)$ and $\psi_0 (z) $.

\noindent Take $  \phi(z) = \int_0^{\infty} \exp(-zt)dG(t).$
Then
\[
\psi_0 (z)  = \int_0^{\infty} \exp(-zt) \exp(-\lambda t) dG(t) =
\phi(z+\lambda),
\]
and
\[
\psi (z)  = \int_0^{\infty} \exp(-zt)  d (-\exp(-\lambda t)(1-G(t))
\]
\[
= \lambda  \int_0^{\infty} \exp(-(z+\lambda )t) (1-G(t))dt +
            \int_0^{\infty}  \exp(-(z+\lambda )t) dG(t)
\]
\[
= \frac{\lambda }{z+\lambda} - \lambda \int_0^{\infty}
\exp(-(z+\lambda )t) G(t)dt + \phi(z+\lambda)
\]
\[
= \frac{\lambda }{z+\lambda}(1- \phi(z+\lambda)) +\phi(z+\lambda) .
\]
Making use of the functions $\psi(z)$ and $\psi_0 (z)$ from the formula (\ref{f6}), we get
\[
\varphi(z)  = \frac{\lambda }{z+\lambda} .
\]
This means that equality (\ref{f16}) is correct.

It is easy to make sure that substitution $P_{S} (t)=\exp(-\lambda t) $ in equation (\ref{f17}) converts it into identity.
\end{proof}

\begin{zau}\label{z1}
If under  conditions of Proposition \ref{l1} the random variables $\epsilon $ and  $\zeta $ are independent, then
\[
F_{\zeta, 0} (t) = (1- q)F_{\zeta} (t) ,    \quad     \psi_0(z)= (1-
q)\psi(z)
\]
and equality (\ref{f6}) is of the form
\begin{equation} \label{f18}
\varphi(z) = \frac{ q \psi(z) }{ 1 - (1-q)\psi(z)} .
\end{equation}
Ofcourse, under this condition, equality (\ref{f18}) can be deduced directly from the definition.

\noindent If, additionally, the random variable $\zeta $ has the exponential distribution,
$F_{\zeta}(t) = 1- \exp(-\lambda t),\ t\ge0 $, then
  $$ \psi(z)= \frac{\lambda}{z+\lambda}, \quad  \varphi(z) = \frac{ q \lambda }{ z + q\lambda} .$$
The last equality means that the random variable  $S_{\nu} $ has the exponential distribution
with parameter $\lambda q$, $$F_S (t) = 1- \exp(-\lambda q t),\ t\ge0.$$
\end{zau}

\begin{zau}\label{z2}
For independent random variables
 ${\nu} $ and  $\zeta $  in the book {\rm{\cite{bor2}}}
the following formula is proposed
\begin{equation} \label{f19}
 \mathbf{D}S_{\nu}  = \mathbf{D}\zeta \mathbf{E}\nu +  (\mathbf{E}\zeta)^2 \mathbf{D}\nu .
\end{equation}
Unfortunately, as follows from Proposition \ref{l1}, for arbitrary Markov moments this simple formula (\ref{f19}) is incorrect.
In general case similar formulas are known only for the values
$\mathbf{E}(S_{\nu} -  \nu \mathbf{E}\zeta )^2 $ (see books {\rm{\cite{bor2}, \cite{kol}}})
\end{zau}

\section{Applications}

\subsection{Application 1. Geiger-Muller counter of type 1.}
Such a counter is used to calculate cosmic particles which
have arrived in some area of space. It acts as follows
(\cite{smith}, p.273,  \cite{fe2}, p.372).
A particle reaching the counter when it is free is registered but locks the counter for a random time
 $ \eta$. Particles reaching the counter during the locked period are not registered.
Denote by $ T $ the time
during which the Geiger-Muller counter will not lose any particles, and let ${P}_{T}(t) = \mathbf{P}(T> t) $.

  To simplify recordings, we assume that at the moment $ t_0 =0 $
a particle reaching the counter, and $ t_1,  t_2, \ldots,  t_k, \ldots $ are the
next moments of particle reaching, and denote $ \tau_k = t_k -
t_{k-1}.  $

Let $ \eta_k  $ be the discharge time of the particle reaching the counter at the moment
 $ t_{k-1}$,  and let $(\tau_k ) $ and $(\eta_k ) $ be independent random sequences,
\[
\mathbf{P}(\tau_k < t) =F(t) , \quad   \mathbf{P}(\eta_k < t) =G(t)
, \quad F(+0)=0,  \quad G(+0)=0 ,
  \]
and
\[
\epsilon_k = I(\tau_k <  \eta_k ) , k \geq 1 , \quad  \nu = \min(k
\geq 1 : \quad \epsilon_{k}= 1).
 \]
Then it is clear that
\begin{equation} \label{f20}
 T  = \sum_{i=1}^{\nu} \tau_i .
\end{equation}
In order to apply Proposition \ref{l1} here, we take
\[
\zeta_i = \tau_i ,\quad i \geq 1 , \quad  F_{\zeta} (t) = F(t) ,
 \]
\[
 F_{\zeta , 0} (t)= \mathbf{P}(\tau_1 < t , \tau_1 > \eta_1 ) = \int_0^t  G(x)
 dF(x) ,
 \]
\[
 q= \mathbf{P}(\tau_1 <  \eta_1 ) = \int_0^{\infty} F(t)
 dG(t) .
 \]
 The equality (\ref{f5}) in this case is of the form
\begin{equation} \label{f21}
P_{T} (t)  = 1-F (t) +   \int_0^{t} P_{T} (t - x)G(x) dF(x).
\end{equation}
The Laplace transform $\varphi(z)$ for distribution of the random variable $ T $ is determined by formula
(\ref{f6}), in which
\[
 \psi(z) = \int_0^{\infty} \exp(-zt)dF(t), \quad \psi_0 (z) = \int_0^{\infty} \exp(-zt)
 G(t)dF(t) .
\]
The mean value of $ \mathbf{E} T $ and the variance
    $\mathbf{D }T $ can be calculated by formulas
(\ref{f7}) and (\ref{f8}).

In the most important case when there is a Poisson flow of
particles with parameter $\lambda$,   $$F(t) = 1- \exp(-\lambda
t) , t\geq 0 ,$$ we have
\[
 \psi(z) = \frac{\lambda}{z+\lambda} , \quad \psi_0 (z) =\frac{\lambda
 \phi(z+\lambda)}{z+ \lambda}
 , \quad \varphi(z) =\frac{\lambda (1-
 \phi(z+\lambda))}{z+ \lambda (1-
 \phi(z+\lambda))}  ,
\]
where
\[
 \phi(z) = \int_0^{\infty} \exp(-zt)dG(t) .
\]
It follows from (\ref{f3}),  (\ref{f7}) and (\ref{f8}) that
\[
\mathbf{E }T = \frac{1}{\lambda q} , \quad \mathbf{D }T =
\frac{2(q+\lambda a_0 )-1}{\lambda^2 q^2},
\]
where
\[
q= 1-\phi(\lambda), \quad  a_0 =\lambda \int_0^{\infty} t
\exp(-\lambda t)G(t) dt.
\]

\subsection{Application 2. Redundant system with renewal.}

The problem of reliability of redundant systems with renewal is devoted to a fairly large number of works,
among which we note the works  \cite{gbc},  \cite{kor}, \cite{kov1}.
Here we  consider a simple renewal case, namely, duplication, when each operating unit is associated with a single standby unit, which
replaces the operating unit in case the latter fails.
The element that failed is repaired, and after repair its characteristics equivalent to the original characteristics.
Suppose that there is one repair unit in the system.
In addition, we consider that the standby unit is in a partially energized state until it is connected instead of the primary unit (light standby).
During the period it is in standby, it can fail but the probability of this is less than the same probability
for the basic unit ( \cite{gbc}, p. 324-330). This is redundant system of type $(iii/2/1) $ in terms of the paper \cite{doma}.
Denote by $\xi(t)$ the total number of defective units in the system
at a moment $t$. We assume that $\xi(0)=0 $ almost surely and
say that the system is in a state $k $ at a
moment $t$, if $\xi (t)=k, k=0, 1, 2 .$
A busy period is understood as a continuous period when the
system is functioning properly (at least one element is functioning). Every busy period is followed by an idle period when both
units fail.

We will assume that the failure-free time of the primary unit $\tau$,
and the failure-free time of the standby unit $\tau'$
have exponential distributions
\begin{equation}
\label{f22}
 \mathbf{P}(\tau<t)=1-e^{-\lambda t},\quad
\mathbf{P}(\tau'<t)=1-e^{-\lambda' t},\quad t\ge0,
\end{equation}
while the renewal time (repairing
time) $\eta $ has an arbitrary distribution
$G(t)=\mathbf{P}(\eta < t),\quad G(0+)=0.$

We suppose that $\tau $,  $\tau'$ and $\eta $ are independent random variables.
Let
\[
q= \mathbf{P}(\tau < \eta) =1- \int_0^{\infty}e^{-\lambda x} dG(x) .
\]
Denote by $W_k $ the k-th busy period.
When the system is in the state $2$ then the system fails,
so $W_1$ is the time to first fail.
For  $k$-th busy period we have simple formulas for the first moments
 (\cite{gbc},  \cite{doma} )
\[
\mathbf{E}W_1= \frac{1}{\lambda } +\frac{1}{(\lambda+ \lambda^{'})
q} \quad \text{òà}\quad
\mathbf{E}W_k=\frac{\lambda+q\lambda'}{\lambda(\lambda+\lambda')q} ,
\quad k\geq 2.
\]
Based on Proposition \ref{l1} we can deduce exact formulas for calculation the variation $\mathbf{D}W_k $.
Really, let   $(\tau_i,\tau'_i,\eta_i )$ be independent copies of $(\tau, \tau',\eta)$, let
\[
\epsilon_i = I(\tau_i <  \eta_i ) , i \geq 1 , \quad  \nu = \min(i
\geq 1 : \quad \epsilon_{i}= 1).
 \]
In the paper \cite{doma} we can find the representation
\begin{equation} \label{f23}
W_1  \stackrel{d} = \sum_{i=1}^{\nu} \zeta_i ,
  \end{equation}
where $\xi_1 \stackrel{d} = \xi_2 $ means equality of distributions of the random variables $\xi_1 $ and  $\xi_2 $.

\noindent In this formula
  $\zeta_i = \min(\tau_i, \eta_i)
+ \tilde{\tau}_i$, random variables  $\tilde{\tau}_i$ do not depend on $\tau_i $ and $ \eta_i $, and so do not depend on $\nu$,
$\mathbf{P}(\tilde{\tau}_i < t) =1- e^{-(\lambda + \lambda^{'}) t},
t\geq 0 .$

\noindent Then by formula (\ref{f8})
\[
 \mathbf{D}W_1   = \frac{\sigma^2 }{q} + \frac{a^2(q-1)+2a a_0 }{q^2
 } ,
 \]
where
\[
a =\mathbf{E}\min(\tau , \eta) +\mathbf{E}\tilde{\tau}=
 \frac{q}{\lambda} + \frac{1}{(\lambda +\lambda^{'})},
  \]
\[
 \sigma^2 =\mathbf{D}\min(\tau , \eta) +\mathbf{D}\tilde{\tau}= \mathbf{E}\min(\tau ,
 \eta)^2 - \frac{q^2}{\lambda^2} + \frac{1}{(\lambda +\lambda^{'})^2}
  \]
\[
 = \frac{2- q^2}{\lambda^2} + \frac{1}{(\lambda +\lambda^{'})^2}
 -2 \int_0^{\infty} t e^{-\lambda  t} G(t) dt,
  \]
\[
a_0 =\mathbf{E}[\min(\tau , \eta) +\tilde{\tau}]I(\tau >  \eta )=
\int_0^{\infty} t e^{-\lambda  t} dG(t) +
 \frac{1-q}{(\lambda +\lambda^{'})}.
  \]
For   $k\geq 2$
\[
 W_1   \stackrel{d} =  W_k + \tilde{\tau},
 \]
(see \cite{doma}), moreover $ W_k $, $ \tilde{\tau}$ are independent. That is why
\[
 \mathbf{D}W_k   = \mathbf{D}W_1 - \frac{1}{(\lambda
 +\lambda^{'})^2}.
 \]
\noindent  B.~V.~Gnedenko \cite{gn1}, \cite{gn2}
found an integral equation  which the function  $\mathbf{P}(W_1 > t )$ satisfies,
 and derived a formula for the Laplace transform of this function.
 Unfortunately, under the general assumptions concerning the distribution $ G (t) $, the explicit form of a function
$\mathbf{P}(W_1 > t )$ is not known.
In the case $ G(t)= 1-e^{-\mu t}, t \geq 0 $, the function $\mathbf{P}(W_1 > t )$ is known (see, for example, \cite{gk}),
but has a rather cumbersome form.
 Therefore, the following equations seem rather unexpected.

Consider the first busy period and denote by
$\alpha_0 $ and $\alpha_1 $  times of stay of the system  in the states
$0 $ and $1 $ correspondingly.
 Therefore
\[
 W_1   =  \alpha_0 + \alpha_1,
 \]
moreover (see \cite{doma})
\[
  \alpha_0   \stackrel{d} = \sum_{i=1}^{\nu} \tilde{\tau}_i , \quad
  \alpha_1 \stackrel{d} = \sum_{i=1}^{\nu} \min(\tau_i , \eta_i ).
 \]
 These representations together with the Corollary \ref{n2} of Proposition \ref{l1} and Remark \ref{z1}
 give the following equalities:
\[
\mathbf{P}( \alpha_0 > t ) = e^{-(\lambda + \lambda^{'})q t}, \quad
\mathbf{P}( \alpha_1 > t ) = e^{-\lambda  t}, \quad t\ge0.
\]
It is clear that random variables $\alpha_0 $ and $\alpha_1 $ are
strongly dependent, which does not allow to find a simple formula for distribution
 $ \mathbf{P}(W_1 > t )$.

\subsection{Application 3. Single-server queueing system}

Consider a single-server queueing system (SSQS)
on the interval $[0, \infty) $.
Let $t_0 =0, t_1 ,\ldots, t_n, \ldots $ be the customer random arrival times.
The arrived customer starts to be served, if the service channel is not occupied.
Otherwise, the customer joins a queue whose size is not limited.
The channel serves the $n$-th customer during a random time  $\eta_n.$
 After the end of the service, the channel or takes the next customer, if the queue is not empty, or is waiting for a new customer.

Suppose that $\tau_{n-1} =t_n - t_{n-1},\, n\geq 1, $ are independent identically distributed random variables with the distribution function
$F(t)=\mathbf{P}(\tau_n < t)= 1 -e^{-\lambda t},  t\ge0. $
It means that there is a Poisson frow of customer arrives to SSQS system (see  \cite{gk}).
Let $\eta_n, n\geq 1, $ be independent identically distributed random variables with the distribution function
$G(t)=\mathbf{P}(\eta_n < t),\, G(+0)=0 .$

Denote by $\xi(t)$ the number of customers in the SSQS system at a moment of time $t$, $\xi(0)=0$ a.s.
We say that at a moment  $t$ SSQS system is in the state $k$, if  $\xi (t)=k.$

We define the regeneration moments $S_{k}$ for the process $\xi(t)$ in the following way.

 Let $S_0 = t_1 $ be the moment when the first customer arrives.
 Let $S_1$ be the first transition moment from the state $2$ to the state $1$.

For $k >1 $ the beginning of $k$-the regeneration cycle $S_{k-1}$ we
take  ${k-1}$-th  transition moment from the state $2$ to the state $1$ while its end $S_{k}$ is ${k}$-th such a moment.

Denote, as in the previous application,
\[
\epsilon_i = I(\tau_i <  \eta_i ), i \geq 1, \quad  \nu = \min(i
\geq 1 : \quad \epsilon_{i}= 1).
 \]

 Figure 1 demonstrates the first regeneration cycle.

\begin{center}
\begin{picture}(370,60)(-10,-30)
\put(0,0){\line(1,0){350}}

\put(0,-3){\line(0,1){6}} \put(30,-5){\line(0,1){10}}
\put(60,-3){\line(0,1){6}} \put(90,-3){\line(0,1){6}}
\put(120,-3){\line(0,1){6}} \put(150,-3){\line(0,1){6}}
\put(200,-3){\line(0,1){6}} \put(230,-3){\line(0,1){6}}
\put(260,-3){\line(0,1){6}} \put(290,-3){\line(0,1){6}}
\put(320,-5){\line(0,1){10}}

\qbezier(0,0)(15,25)(30,0)
\qbezier(30,0)(40,13)(60,13)\multiput(60,12)(0,-3){3}{\line(0,-1){1}}
\qbezier(60,0)(75,25)(90,0)
\qbezier(90,0)(100,13)(120,13)\multiput(120,12)(0,-3){3}{\line(0,-1){1}}
\qbezier(120,0)(135,25)(150,0)

\qbezier(200,0)(210,13)(230,13)\multiput(230,12)(0,-3){3}{\line(0,-1){1}}
\qbezier(230,0)(245,25)(260,0) \qbezier(260,0)(275,25)(290,0)

\put(170,5){\ldots}  \put(300,5){\ldots} \put(335,5){\ldots}

\qbezier(30,0)(45,-25)(60,0) \qbezier(90,0)(105,-25)(120,0)
\qbezier(200,0)(215,-25)(230,0) \qbezier(260,0)(290,-22)(295, -10)

\put(12,20){${\tau}_0$} \put(42,20){$\tau_1$}
\put(72,20){$\widetilde{\tau}_1$} \put(102,20){$\tau_2$}
\put(132,20){$\widetilde{\tau}_2$}

\put(212,20){$\tau_{\nu-1}$}
\put(242,20){$\widetilde{\tau}_{\nu-1}$} \put(272,20){$\tau_\nu$}

\put(-7,-30){t=0} \put(25,-30){$S_0$} \put(43,-25){$\eta_1$}
\put(103,-25){$\eta_2$}\put(213,-25){$\eta_{\nu-1}$}\put(288,-25){$\eta_\nu$}
\put(317,-30){$S_1$}
\end{picture}
\end{center}
\begin{center}
Figure 1.
\end{center}

The random variables $\widetilde{\tau}_i $ on Figure 1 do not depend on  $\tau_i  $ and
$\eta_i $,  and $\mathbf{P}(\widetilde{\tau}_i <t) =1- e^{-\lambda
t}$.

We describe the work of the SSQS in the first period of regeneration.
Let $T = S_1 - S_0 $ be its duration, let $S_0 + \alpha $ be the first
transition moment to the state $2$.
On the interval  $ [S_0 , S_0 + \alpha ) $ the system can only be in states $0$ or $1$, the queue is absent.
 Next is the interval $ [ S_0 + \alpha , S_0 + \alpha +\beta)$ of length $\beta$, on which the queue is always $\geq 1 $,
 \begin{equation} \label{f24}
T = \alpha+ \beta  .
 \end{equation}
We find some basic characteristics of the random variables $\alpha $ and $\beta$.

\noindent Denote by
 ${\alpha}_0$  and    ${\alpha}_1$ be sojourn times in states $0$ and   $1$ on $[S_0,S_1)$.
 It is clear that
\begin{equation} \label{f25}
\alpha =\alpha_0 + \alpha_1 .
  \end{equation}
 We can find distributions of random variables  ${\alpha}_0$ and   ${\alpha}_1$ based on results of Proposition \ref{l1}
\begin{eqnarray}\label{f26}
&&\nonumber \mathbf{P}(\alpha_0 > t) =(1-q) e^{-\lambda qt}, \, t>0
 ,\\
&&\mathbf{P}(\alpha_0 =0) =q .
\end{eqnarray}
\begin{eqnarray}\label{f27}
\mathbf{P}(\alpha_1 > t) = e^{-\lambda t}, \, t>0 .
\end{eqnarray}
Proof of equality (\ref{f26}). We can conclude from Figure 1 that
\[
 {\alpha}_0
=\sum_{i=1}^{\nu-1}\widetilde{\tau}_i .
\]
As we noted before the random variable  $\nu $ and the sequence $(\widetilde{\tau}_i )$ are independent, $\widetilde{\tau}_i $
have the exponential distribution with parameter $\lambda$, that is why
\[
\mathbf{E}\exp(-z\alpha_0 ) =\sum_{k=1}^{\infty}q (1-q)^{k-1}
 [\mathbf{E}\exp(-z\widetilde{\tau}_1)]^{k-1} =\sum_{k=1}^{\infty}q (\frac{\lambda(1-q)}{z+\lambda})^{k-1}
 = \frac{q(z+\lambda)}{z+q\lambda}.
\]
It is also not difficult to check that the distribution function
\begin{eqnarray}
  H(t) &= & \left \{
 \begin{array}{rl}
 0, & \mbox{ïðè} \quad t < 0 , \\
1- (1-q) \exp (-\lambda qt), & \mbox{ïðè} \quad  t> 0 ,  \,
  \end{array} \right.  \nonumber
\end{eqnarray}
satisfies the equality
 \[
  \int_0^{\infty} \exp(-zt)dH(t) = \frac{q(z+\lambda)}{z+q\lambda}.
 \]
This means that equality  (\ref{f26}) is correct.

\noindent  Let's move to equality (\ref{f27}). It follows from Figure 1
that
\begin{eqnarray}
\quad {\alpha}_1 =\sum_{i=1}^\nu\min(\tau_i,\eta_i). \nonumber
\end{eqnarray}
So (\ref{f27}) follows from Corollary \ref{n1} of Proposition \ref{l1}.

In the following we will assume that
\[
 b= \mathbf{E}\eta_i < \infty , \quad   \rho = \lambda b < 1 .
 \]
It is known \cite{gk} that in this case there exist stationary probabilities of states
\[
\lim_{t \to\infty} \mathbf{P}(\xi(t) = i) = p_i ,  \quad i = 0, 1,
2, \ldots ,
 \]
and
\begin{equation}\label{f28}
 p_0 = 1- \rho .
\end{equation}
From (\ref{f26}), (\ref{f27}) we get
\begin{equation}\label{f29}
 \mathbf{E}\alpha_0 =   \frac{1-q }{\lambda q} ,  \quad  \mathbf{E}\alpha_1 =   \frac{1}{\lambda }.
\end{equation}
We need also equalities
\begin{equation}\label{f30}
 p_0 =   \frac{\mathbf{E}\alpha_0 }{\mathbf{E}T } ,  \quad  p_1 =   \frac{\mathbf{E}\alpha_1 }{\mathbf{E}T
 },
\end{equation}
(see \cite{doma}).

\noindent From equalities  (\ref{f24}), (\ref{f28}) -- (\ref{f30}) it follows that
\[
\mathbf{E}\alpha = \mathbf{E}\alpha_0 + \mathbf{E}\alpha_1
=\frac{1}{\lambda q},
 \]
\[
\mathbf{E}T = \frac{\mathbf{E}\alpha_0}{p_0} =\frac{1-q }{\lambda
q(1-\rho)},
 \]
\[
\mathbf{E}\beta =\frac{\rho -q }{\lambda q(1-\rho)}.
 \]
We also note that from the derived equations the following simple formula  follows
\[
p_1 = \frac{q(1-\rho)}{1-q}.
 \]

\section{Conclusions}

We propose results of the investigation of properties of the random sums of random variables under the condition that the random summands (i.i.d. random variables) are not independent
on the (random) number of summands.
We consider the case, where the number of summands is the first moment of an event occurrence.
We propose integral equation and some relations that determine distributions of random sums, their Laplace transforms and main characteristics (first and second moments).
Some applications of the obtained results are described. We, first, analyse the distribution function of the time
during which the Geiger-Muller counter of type 1 does not lose any particles.
Then, the distribution function and its main characteristics of the busy period of a redundant system with renewal is nalysed.
And, third, the distribution function of the sojourn times of a single-server queueing system is analysed.

\end{document}